\documentclass[11pt,twoside]{amsart}
\usepackage{amsmath, amsthm, amscd, amsfonts, amssymb, graphicx}
\usepackage{enumerate}
\usepackage[colorlinks=true,
linkcolor=blue,
urlcolor=cyan,
citecolor=red]{hyperref}
\usepackage{mathrsfs}
\newtheorem{theorem}{Theorem}[section]
\newtheorem{lemma}{Lemma}[section]
\newtheorem{remark}{Remark}[section]

\newtheorem{corollary}{Corollary}[section]

\newtheorem{proposition}{Proposition}[section]
\numberwithin{equation}{section}

\begin{document}
\title{Numerical radii of accretive matrices}
\author{Yassine Bedrani, Fuad Kittaneh and Mohammed Sababheh}
 \maketitle

\pagestyle{myheadings}
\markboth{\centerline {}}
{\centerline {}}
\bigskip
\bigskip
\begin{abstract}
The numerical radius of a matrix is a scalar quantity that has many applications in the study of matrix analysis. Due to the difficulty in computing the numerical radius, inequalities bounding it have received a considerable attention in the literature. In this article, we present many new bounds for the numerical radius of accretive matrices. The importance of this study is the presence of a new approach that treats  a specific class of matrices, namely the accretive ones.  The new bounds provide a new set of inequalities, some of which can be considered as refinements of other existing ones, while others present new insight to some known results for positive matrices.
\end{abstract}
\section{Introduction}
Let $\mathcal{M}_n$ be the algebra of all complex $n\times n$ matrices. For $A\in\mathcal{M}_n$, the numerical radius $w(A)$ and the operator norm $\|A\|$ of $A$ are defined, respectively, by
$$w(A)=\max\{|\left<Ax,x\right>|:x\in\mathbb{C}^n, \|x\|=1\}$$ and
$$\|A\|=\max\{|\left<Ax,y\right>|:x,y\in\mathbb{C}^n, \|x\|=\|y\|=1\}.$$
It is well known that $w(\cdot)$ defines a norm on $\mathcal{M}_n$ that is equivalent to the operator norm, via the relation \cite[p 114.]{Halmos}
\begin{equation}\label{eq_num_oper}
\frac{1}{2}\|A\|\leq w(A)\leq \|A\|, A\in\mathcal{M}_n.
\end{equation}
Interest in bounding the numerical radius has grown due to the fact that computing the operator norm is much easier than that of the numerical radius. For this reason, we find many research papers presenting bounds for $w(A)$ in terms of the operator norm. 

The numerical range of $A\in \mathcal{M}_n$ is defined by the set
$$W(A)=\{\left<Ax,x\right>:x\in\mathbb{C}^n,\|x\|=1\}.$$ 
If $W(A)\subset (0,\infty),$ we say that $A$ is positive, and we simply write $A> 0.$ It is well known that when $A> 0,$ we have $w(A)=\|A\|.$ A more general class of matrices than that of positive ones is the so called accretive matrices. A matrix $A\in\mathcal{M}_n$ is said to be accretive when $\Re A>0.$ Notice that 
$$\Re A>0\Leftrightarrow W(A)\subset (0,\infty)\times (-\infty,\infty)\subset \mathbb{C},$$  where $\Re A=\frac{A+A^*}{2}$ is the real part of $A$. It is clear that when $A$ is positive, it is necessarily accretive.
 
 The main goal of this article is to present many new relations for $w(A)$ when $A$ is accretive. Some of these new forms present a new direction in this study, while others can be looked at as refinements of some known results, in a new setting.
 
When talking about accretive matrices, we need to introduce sectorial matrices.  A matrix $A\in\mathcal{M}_n$ is said to be sectorial if, for some $0\leq \alpha<\frac{\pi}{2}$, we have
$$W(A)\subset S_{\alpha}:=\{z\in\mathbb{C}:|\Im z|\leq \tan\alpha\; \Re z\}.$$ The smallest such $\alpha$ will be called the sectorial index of $A$. When $W(A)\subset S_{\alpha},$ we will write $A\in \mathcal{S}_{\alpha}.$	Further, in the sequel, it will be implicitly understood that the notion of $S_{\alpha}$ and $\mathcal{S}_{\alpha}$ are defined only when $0\leq \alpha<\frac{\pi}{2}.$

Recently, in \cite{bedr} the operator mean of two accretive matrices $A,B\in \mathcal{M}_n$ has been defined by
 \begin{align}\label{eq_def_sigma}
 A\sigma_f B=\int^1_0 (A!_sB)\;d\nu_f (s),
\end{align}   
where $A!_sB=((1-s)A^{-1}+sB^{-1})^{-1}$ is the harmonic mean of $A,B$, the function $f:(0,\infty)\longrightarrow(0,\infty)$ is an operator monotone function with $f(1)=1$ and $\nu_f$ is the probability measure characterizing $\sigma_f$.

Moreover, they also characterize the operator monotone function for an accretive matrix: let $A\in\mathcal{S}_{\alpha}$ and $f:(0,\infty)\longrightarrow(0,\infty)$ be an operator monotone function with $f(1)=1,$ then

\begin{align}\label{eq_def_f(A)}
f(A)=\int^1_0 ((1-s)I+sA^{-1})^{-1}\;d\nu_f (s),
\end{align}

where $\nu_f$ is probability measure satisfying $f(x)=\int^1_0 ((1-s)+sx^{-1})^{-1}\;d\nu_f (s).$

This definition was motivated by the same definition for positive matrices, and many properties of operator mean of accretive matrices were given in \cite{bedr}.  In \cite{ftan},  the logarithmic mean of accretive $A,B$ is defined by
\begin{align}\label{loga_matrices}
\mathcal{L}(A,B)=\int^{1}_{0}A\sharp_t B\;dt.
\end{align}
The Heinz mean is defined in \cite{YMao} as 
\begin{align}\label{heinz_for_matrices}
\mathcal{H}_t(A,B)=\dfrac{A\sharp_t B+A\sharp_{1-t} B}{2},\;0\leq t\leq 1,
\end{align}
where $\sharp_t$ stands for the weighted geometric mean, which corresponds to the operator monotone function $f(x)=x^t,0\leq t\leq 1.$

When $A\in\mathcal{S}_0,$ we have $w(A)=\|A\|.$ Our first simple observation will be that when $A\in\mathcal{S}_{\alpha},$ we have 
$$\cos\alpha\;\|A\|\leq w(A)\leq \|A\|.$$ Notice that this new inequality is better than \eqref{eq_num_oper} when $0\leq \alpha< \frac{\pi}{3}.$ Many extensions of some numerical radius inequalities will be shown for accretive and sectorial matrices, including power inequalities and sub multiplicative behavior.

Another set of new inequalities for accretive matrices is the treatment of $w(f(A))$ and $w(A\sigma B)$, where $f$ is an operator monotone function and $\sigma$ is an operator mean. Such inequalities have not been treated in the literature due to the fact that when $A,B$ are positive, $f(A)$ and $A\sigma B$ are positive, and hence their numerical radius and operator norms coincide. So, when $A,B$ are accretive, this presents a new direction.

Many other results will be presented, like subadditivity of the numerical radius, relations among $w(A)$ and $w(\Re A)$, and many others.

 For our purpose, we will need the following  notation . 
$$ \mathfrak{m}=\{f(x)\;{\text{where}}\;f:(0,\infty)\to (0,\infty)\; {\text{is an operator monotone function with}} \;f(1)=1\}.$$

\section{Some preliminary discussion}
In this part of the paper, we discuss some needed results and terminologies related to accretive matrices.
\begin{lemma} \cite{bedr}\label{lemma_real_a_sigma_b_less}
Let $ A, B\in\mathcal{S}_{\alpha} $. If $f\in \mathfrak{m},$ then
\begin{equation}
\Re(A\sigma_f B)\leq \sec^{2}\alpha\;(\Re A)\;\sigma_f\;(\Re B).
\end{equation}
\end{lemma} 
\begin{lemma}\label{lemma_ando_zhan}\cite{ando_zhan} Let $A, B\in \mathcal{M}_n$ be two positive  matrices. Then, for any  non-negative operator monotone function $f$ on $\left[0,\infty  \right)$,
\begin{equation}
|||f(A+B)|||\leq |||f(A)+f(B)|||
\end{equation}
\end{lemma}
\begin{lemma}\cite{Ando_2}\label{sigma_norm} Let $A, B\in\mathcal{M}_n$ be positive. If $f\in \mathfrak{m},$ then 
\begin{align}
||| A \sigma_f B|||\leq ||| A|||\sigma_f||| B|||.
\end{align}
\end{lemma}
\begin{lemma}\cite{bedr}\label{lemma_f_real_sec_f} Let $ A\in\mathcal{S}_\alpha. $ If $f\in \mathfrak{m},$ then 
\begin{equation}
f(\Re A)\leq\Re (f(A))\leq\ \sec^{2}\alpha \;f(\Re A)
\end{equation}
\end{lemma}
\begin{lemma}\cite{bedr}\label{realf_fnorm} Let $ A\in\mathcal{S}_\alpha. $ If $f\in \mathfrak{m},$ then 
\begin{align*}
f(\|\Re A\|)\leq\|\Re f(A)\|\leq\sec^2\alpha\;f(\|\Re A\|).
\end{align*}
\end{lemma}
\begin{lemma}\cite{Choi}\label{negtive_power_of_ real} Let $A\in\mathcal{S}_\alpha$ and $t\in[-1,0]$. Then
\begin{align}
\Re A^t\leq \Re^t A \leq \cos^{2t}\alpha\;\Re A^t
\end{align}
\end{lemma}
A reverse of Lemma \ref{negtive_power_of_ real} is as follows.
\begin{lemma}\cite{Choi}\label{positive_power_of_ real} Let $A\in\mathcal{S}_\alpha $ and $t\in[0,1]$. Then
\begin{align}
\cos^{2t}\alpha\;\Re A^t\leq \Re^t A \leq \Re A^t
\end{align}
\end{lemma}
It is well known that for any matrix $A\in\mathcal{M}_n$, $|||\Re A|||\leq |||A|||$, for any unitarily invariant norm $\|\cdot\|$ on $\mathcal{M}_n$. The following lemma presents a reversed version of this inequality for sectorial matrices.
\begin{lemma}\cite{Zhang}\label{norm}
Let $ A \in \mathcal{S}_{\alpha} $ and let $ \parallel.\parallel $ be any unitarily invariant norm on $\mathcal{M}_n$. Then
\begin{center}
$ \cos\alpha\; ||| A||| \leq\
 ||| \Re(A)||| \leq |||A|||.$
\end{center}
\end{lemma}
\begin{lemma} \cite{kitt1}\label{nume_real<nemu A} Let $A\in\mathcal{M}_n$. Then
\begin{align}
w(\Re A)\leq w(A).
\end{align}

\end{lemma}
\begin{lemma}\cite{YMao} \label{heinz_norm_bound} Let $ A, B \in \mathcal{S}_{\alpha} $. Then for $t\in (0,1)$, 
\begin{align}
\cos^3\alpha\;|||A\sharp B|||\leq |||\mathcal{H}_t(A,B)|||\leq\dfrac{\sec^3\alpha}{2}\;|||A+B|||.
\end{align}
\end{lemma}
\begin{lemma}\cite{drury1}\label{S_alphat} Let $ A \in \mathcal{S}_{\alpha} $ and $t\in (0,1)$. Then $ W(A^t)\in S_{t\alpha}$.

Also note that $ W(A^{-t})\in S_{t\alpha}$, by the result indicates that if $W(A)\in S_{\alpha}$ then $W(A^{-1})\in S_{\alpha}$ .   
\end{lemma}

\begin{lemma} \cite{abo_omar}\label{abu_omar_kittaneh} Let $ A ,B\in \mathcal{M}_{n} $ be positive matrices. Then 

\begin{align}
w\left[ \begin{pmatrix}
0&A\\
B&0
\end{pmatrix} 
\right] = \dfrac{1}{2}\|A+B\|.
\end{align}
\end{lemma}

The following two lemmas are well known.
\begin{lemma} \label{max_norm} Let $ A ,B\in \mathcal{M}_{n} $. Then 

\begin{align}
\left\| \begin{pmatrix}
0&A\\
B&0
\end{pmatrix} 
\right\| = \max(\|A\|,\|B\|).
\end{align}
\end{lemma}

\begin{lemma} \label{lem_negative_power_norm} Let $ A\in \mathcal{M}_{n} $ be invertible . Then

\begin{align}
\|A\|^{-1}\leq \|A^{-1}\|.
\end{align}

\end{lemma}
\begin{lemma}\cite{kubo_ando}\label{monoto_mean} Let $A,B,C,D\in\mathcal{M}_n$ be positive. Then

$A\leq C$ and $B\leq D$ imply $A\sigma B\leq C\sigma D$.
\end{lemma} 
\section{Main results}
Now we are ready to present our results. We will present our results in three subsections. In the first subsection, we present inequalities for the numerical radii of accretive matrices that extend some well known inequalities for the numerical radius. However, in the second subsection, we present a new type of numerical radius inequalities that has never been tickled in the literature. The last subsection treats inequalities for the numerical radius and it's connection to operator means.

\subsection{Accretive versions of some known numerical radius inequalities}
First, we have the simple accretive version of \eqref{eq_num_oper}.
\begin{proposition} \label{prop_1}
Let $ A \in \mathcal{S}_{\alpha} $. Then 
\begin{align}
\cos\alpha\;\|A\|\leq w(A)\leq \|A\|
\end{align}
\end{proposition}
\begin{proof} 
Noting that $w(\Re A)=\|\Re A\|$, since $\Re A>0,$ Lemma \ref{norm} implies 
\begin{align*}
\cos\alpha\;\|A\|\leq \|\Re A\|=w(\Re A)\leq w(A)\leq\|A\|.
\end{align*}

\end{proof}

\begin{remark}
Notice that when $0<\alpha<\frac{\pi}{3},$ $\cos\alpha>\frac{1}{2}.$ This means that, for such $\alpha$, 
$$\frac{1}{2}\|A\|< \cos\alpha\; \|A\|\leq w(A)\leq \|A\|,$$ providing a considerable refinement of the left inequality in \eqref{eq_num_oper}.
\end{remark}
\begin{corollary} Let $ A \in \mathcal{S}_{\alpha} $. Then for $t\in (-1,1)$,
\begin{align}
\cos t\alpha\;\|A^t\|\leq w(A^t)\leq \|A^t\|.
\end{align}
\end{corollary}
\begin{proof} Proposition \ref{prop_1}, Lemma \ref{norm} and Lemma \ref{S_alphat}, imply the desired result.
\end{proof} 
While $w(\Re A)\leq w(A)$ for any matrix $A$, a reversed version can be found via sectorial matrices, as follows.
\begin{corollary} Let $ A \in \mathcal{S}_{\alpha} $. Then
\begin{align}\label{inv_w(rA)}
w(A)\leq \sec\alpha\;w(\Re A).
\end{align}

\end{corollary}
\begin{proof} Let $A\in \mathcal{S}_{\alpha}$. Then $w(\Re A)=\|\Re A\|$, since $\Re A>0.$ Proposition \ref{prop_1} implies
\begin{align*}
w(A)\leq \|A\|\leq \sec\alpha\;\|\Re A\|=\sec\alpha\; w(\Re A).
\end{align*}

\end{proof}

In the next results, we present accretive versions of the well known power inequality \cite{Halmos} 
\begin{align}\label{lemma_powers}
w(A^k)\leq w^k(A), A\in\mathcal{M}_n, k=1,2,\cdots.
\end{align}
It should be noted that in \eqref{lemma_powers}, only positive integer powers are treated. Now we add the interval $(0,1)$ to these powers. The significance of these results is the observation that when $A$ is positive, $w(A^t)=\|A^t\|$ for any $t\in (0,1).$ For such powers, we find no version of  \eqref{lemma_powers} in the literature. Now we have one that reads as follows. 

\begin{theorem} Let $ A \in \mathcal{S}_{\alpha} $. Then, for $t\in(0,1),$
\begin{align}
 \cos t\alpha\;\cos^t\alpha\;w^t(A)\leq w(A^t)\leq \sec t\alpha\;\sec^{2t}\alpha\;w^t(A).
\end{align}

\end{theorem}
\begin{proof}
Let $t\in(0,1).$ Then
\begin{align*}
w(A^t)\leq \|A^t\|&\leq\sec t\alpha\;\|\Re A^t\| \hspace{1cm}\text{(by Lemma \ref{norm})}\\
&\leq \sec t\alpha\;\sec^{2t}\alpha\;\|\Re^t A\|\hspace{1cm}\text{(by Lemma \ref{positive_power_of_ real})}\\
&= \sec t\alpha\;\sec^{2t}\alpha\;\|\Re A\|^t\\
&=\sec t\alpha\;\sec^{2t}\alpha\;w^t(\Re A)\\
&\leq \sec t\alpha\;\sec^{2t}\alpha\;w^t(A).\hspace{1cm}\text{(by Lemma \ref{nume_real<nemu A})}
\end{align*}
Thus, we have shown the second inequality. To show the first inequality, we have 
\begin{align*}
w(A^t)\geq \cos t\alpha\;\|A^t\|&\geq\cos t\alpha\;\|\Re A^t\| \hspace{1cm}\text{(by Lemma \ref{norm})}\\
&\geq \cos t\alpha\;\|\Re^t A\|\hspace{1cm}\text{(by Lemma \ref{positive_power_of_ real})}\\
&=\cos t\alpha\;\|\Re A\|^t\\
&\geq\cos t\alpha\;\cos^t \alpha\;\|A\|^t\hspace{1cm}\text{(by Lemma \ref{norm})}\\
&\geq\cos t\alpha\;\cos^t \alpha\;w^t(A),\hspace{0.7cm}\text{(by Lemma \ref{eq_num_oper})}.
\end{align*}
This completes the proof.

When $A$ are positive, then $\alpha=0,$ and we obtain the well known equality $\|A^t\|=\|A\|^t$.

\end{proof}
On the other hand, a negative-power version of  \eqref{lemma_powers} can be stated as follows.
\begin{theorem} Let $ A \in \mathcal{S}_{\alpha} $. Then, for $t\in[0,1],$
\begin{align}\label{ine_negative_power}
 \cos t\alpha\;\cos^{2t} \alpha\;w^{-t}(A)\leq w(A^{-t}).
\end{align}
\end{theorem}

\begin{proof} For such $t\in[0,1]$, we have
\begin{align*}
w(A^{-t})\geq \cos t\alpha\;\|A^{-t}\|&\geq \cos t\alpha\;\|\Re A^{-t}\| \hspace{1.2cm}\text{(by Lemma\; \ref{norm})}\\
&\geq \cos t\alpha\;\cos^{2t}\alpha\;\|\Re^{-t} A\|\hspace{1cm}\text{(by Lemma\; \ref{negtive_power_of_ real})}\\
&\geq\cos t\alpha\;\cos^{2t}\alpha\;\|\Re A\|^{-t}\hspace{1.2cm}\text{(by Lemma\;\ref{lem_negative_power_norm})}\\
&=\cos t\alpha\;\cos^{2t}\alpha\;w^{-t}(\Re A)\\
&\geq \cos t\alpha\;\cos^{2t}\alpha\;w^{-t}(A),\hspace{1cm}\text{(by Lemma\; \ref{nume_real<nemu A})}
\end{align*}
completing the proof.

When $A$ are positive, then $\alpha=0,$ and we obtain the well known inequality $\|A\|^{-t}\leq\|A^{-t}\|$, for $t\in[0,1].$
\end{proof}
In particular, we have the following interesting inverse relation. It should be noted that in general, we have no relation between $w^{-1}(A)$ and $w(A^{-1}).$ Now we have the following accretive version.
\begin{corollary} \label{wA^-1} Let $ A\in S_{\alpha}. $  Then
\begin{align}
 \cos^{3}\alpha\;w^{-1}(A)\leq w(A^{-1}).
\end{align}
\end{corollary}
\begin{proof} Let $t=1$ in \eqref{ine_negative_power}.
\end{proof}

In the following result, we present a new submultiplicative inequality for the numerical radius. Recall that for general $A,B\in\mathcal{M}_n$, one has $w(AB)\leq 4w(A)w(B).$ When $A$ and $B$ commute, the factor 4 can be reduced to 2, while it can be reduced to 1 when $A$ and $B$ are normal \cite[p 114.]{Halmos}. The following result presents a new bound, that is better than these bounds for  $0<\alpha<\frac{\pi}{3}.$
\begin{theorem} \label{nemu_for_AB}Let $ A, B \in\mathcal{S}_{\alpha}$. Then
\begin{align}
w(AB)\leq \sec^2\alpha\;w(A)w(B).
\end{align}
\end{theorem}
\begin{proof} We have 
\begin{align*}
w(AB)&\leq\|AB\|\hspace{2.4cm}\text{(by Lemma\; \ref{eq_num_oper})}\\ 
&\leq \|A\| \|B\|\\
&\leq \sec^2\alpha\; \|\Re A\| \|\Re B\|\hspace{1.2cm}\text{(by Lemma\; \ref{norm})}\\
&=\sec^2\alpha\; w(\Re A)w(\Re B)\\
&\leq\sec^2 \alpha\;w(A)w(B),\hspace{1.2cm}\text{(by Lemma\; \ref{nume_real<nemu A})}\\
\end{align*}
which completes the proof.

When $A,B$ are positive, then $\alpha=0,$ and we obtain the well known inequality $w(AB)\leq w(A)w(B).$
\end{proof}

\subsection{The numerical radius and operator monotone functions}
A new type of numerical radius inequalities is discussed then, when relations for $w(f(A))$ and $f(w(A))$ are found. However, for such inequalities to be studied, we prove first that when $A\in\mathcal{S}_{\alpha}$ and $f\in\mathfrak{m}$, then $f(A)\in\mathcal{S}_{\alpha}.$ This follows from the following.

\begin{proposition}\label{prop_asigmab_s}
Let $A,B\in\mathcal{S}_{\alpha}$ and let $f\in\mathfrak{m}.$ Then $A\sigma_fB\in \mathcal{S}_{\alpha}.$
\end{proposition}

\begin{proof}
Let $A,B\in S_{\alpha}$ and notice that\cite[Definition 4.1]{bedr}
$$A\sigma_fB=\int_{0}^{1}(A!_{s}B)d\nu_f(s)\;\;({\text{see}}\;\eqref{eq_def_sigma})$$ for some positive measure $\nu_f(s)$ on $[0,1].$ Then for any unit vector $x\in\mathbb{C}$, we have
\begin{align*}
\left<A\sigma_fBx,x\right>&=\int_{0}^{1}\left<(A!_sB)x,x\right>d\nu_f(s)\\
&=\int_{0}^{1}h(s)d\nu_f(s)\;({\text{where}}\;h(s)=\left<(A!_sB)x,x\right>)\\
&=c+id,
\end{align*}
where
$$c=\Re \int_{0}^{1}h(s)d\nu_f(s), d=\Im \int_{0}^{1}h(s)d\nu_f(s).$$ We notice that for each $s\in [0,1]$, $h(s)\in S_{\alpha}$ since $A,B\in S_{\alpha}.$ This is due to the fact that $\mathcal{S}_{\alpha}$ is invariant under inversion and addition. To show that $A\sigma_fB\in \mathcal{S}_{\alpha},$ we need to show that $\left<A\sigma_fB)x,x\right>\in S_{\alpha},$ or $|d|\leq \tan (\alpha)c.$ In fact, we have
\begin{align*}
|d|&=\left|\Im \int_{0}^{1}h(s)d\nu_f(s)\right|\\
&\leq \int_{0}^{1}\left|\Im h(s)\right|d\nu_f(s)\\
&\leq \int_{0}^{1}\tan(\alpha)\Re h(s)d\nu_f(s)\;\;(\text{since}\;h(s)\in S_{\alpha})\\
&=\tan(\alpha)c.
\end{align*}
This shows that $A\sigma_fB\in \mathcal{S}_{\alpha}$ and completes the proof.
\end{proof}

Noting \eqref{eq_def_f(A)}, we have
\begin{align*}
I\sigma_f A=f(A), f\in\mathfrak{m}.
\end{align*} 
 Then Proposition \ref{prop_asigmab_s} implies the following.

\begin{corollary}\label{cor_f(A)_sect}
Let $A\in\mathcal{S}_{\alpha}$ and $f\in\mathfrak{m}$. Then $f(A)\in\mathcal{S}_{\alpha}.$
\end{corollary}

As a special case, we have the following.
\begin{corollary}\label{cor_at_s}
Let $A\in \mathcal{S}_{\alpha}$ and $t\in (0,1)$. Then $A^t\in \mathcal{S}_{\alpha}.$
\end{corollary}
It should be noted that in \cite{drury1}, it is shown that if $A\in\mathcal{S}_{\alpha},$ then $A^t\in\mathcal{S}_{t\alpha}, t\in (0,1),$ a stronger version of Corollary \ref{cor_at_s}.\\
Now we are ready to present the following new relation that allows switching the numerical radius and  operator monotone functions.
\begin{theorem} Let $ A\in\mathcal{S}_{\alpha}$. If $ f\in\mathfrak{m},$  then
\begin{align}\label{w(f)_leq_f(w)}
\cos\alpha\;f(w(A))\leq w(f(A))\leq \sec^3\alpha\;f(w(A)).
\end{align}
\end{theorem}
\begin{proof} First we note that for every nonegative monotone function $f$ and every $0\leq s\leq 1$, one can get $f(sx)\geq sf(x)$. next we estimate the first inquality
\begin{align*}
\cos\alpha\;f(w(A))&\leq f(\cos\alpha\;w(A))\\
&\leq f(w(\Re A))\hspace{2cm}\text{(by \eqref{inv_w(rA)})}\\
&=f(\|\Re A\|)\\
&\leq\|\Re f(A)\|\hspace{2cm}\text{(by Lemma\; \ref{realf_fnorm})}\\
&=w(\Re f(A))\leq w(f(A)).
\end{align*}
Thus, we have shown the first inequality. To show the second inequality, noting Corollary \ref{cor_f(A)_sect}, we have
\begin{align*}
w(f(A))\leq \|f(A)\|&\leq \sec\alpha\;\|\Re f(A)\|\hspace{1.2cm}\text{(by Lemma\; \ref{norm})}\\
&\leq \sec^3\alpha\; f(\|\Re A\|)\hspace{1.2cm}\text{(by Lemma \ref{realf_fnorm})}\\
&=\sec^3\alpha\; f(w(\Re A))\\
&\leq \sec^3\alpha\; f(w(A)),
\end{align*}
where we have used the fact that $f$ is monotone to obtain the last inequality. 
This completes  the proof.

When $A$ are positive, then $\alpha=0,$ and we obtain the well known inequality $\|f(A)\|= f(\|A\|).$
\end{proof}

\begin{proposition} \label{w_concavity} Let $ A, B\in\mathcal{S}_{\alpha}$. If $ f\in\mathfrak{m},$  then for $\lambda\in (0,1)$,
\begin{align}
w((1-\lambda)f(A)+\lambda f(B))\leq \sec^3\alpha\;f((1-\lambda)w(A)+\lambda w(B)).
\end{align}
\end{proposition}

\begin{proof} We have 

\begin{align*}
w((1-\lambda)f(A)+\lambda f(B))&\leq (1-\lambda)w(f(A))+\lambda w(f(B))\\
&\leq \sec^3\alpha\;\left( (1-\lambda)f(w(A))+\lambda f(w(B))\right) \hspace{2cm} \text{(by \eqref{w(f)_leq_f(w)})}\\
&\leq \sec^3\alpha\; f((1-\lambda)w(A)+\lambda w(B)), 
\end{align*}
where we have used the fact that $f$ is concave to obtain the last inequality.
This completes the proof.
\end{proof}

\begin{corollary} \label{exm} Let $ A, B\in\mathcal{S}_{\alpha}$. Then, for $0<t<1,$  
\begin{align*}
w(A^t +B^t)\leq 2^{1-t}\sec^3\alpha\;\left( w(A)+w(B)\right)^t. 
\end{align*} 
\end{corollary}
\begin{proof} In Proposition \ref{w_concavity}, let $f(x)=x^t, \; t\in (0,1)$ and  $\lambda=\dfrac{1}{2}$. 
\end{proof}

On the other hand, a subadditive inequality for the numerical radius with operator monotone functions is shown next. This inequality is the numerical radius version of the celebrated result stating that
$$|||f(A+B)|||\leq |||f(A)+f(B)|||, A,B\geq 0, f\in\mathfrak{m},$$ shown in \cite{ando_zhan} for any unitarily invariant norm $|||\cdot|||$ on $\mathcal{M}_n.$ Now we present the numerical radius version of this inequality, noting that the numerical radius is not a unitarily invariant norm.
\begin{theorem}\label{nemu_f} Let $ A, B\in\mathcal{S}_{\alpha}$. If $ f\in\mathfrak{m},$  then
\begin{align}
w(f(A+B))\leq \sec^3\alpha\;\left( w(f(A))+w(f(B)\right) .
\end{align}
\end{theorem}

\begin{proof} We have the following chain of inequalities
\begin{align*}
w\left( f\left( A+B\right) \right)\leq \| f\left( A+B\right)\|&\leq \sec\alpha\;\|\Re f\left( A+B\right)\|\hspace{1.2cm}\text{(by Lemma\; \ref{norm})}\\
&\leq \sec^3\alpha\;\| f\left( \Re A+\Re B\right)\|\hspace{1.2cm}\text{(by Lemma\; \ref{lemma_f_real_sec_f})}\\
&\leq \sec^3\alpha\;\| f\left(\Re A\right) +f\left(\Re B\right)\|\hspace{1.2cm}\text{(by Lemma\; \ref{lemma_ando_zhan})}\\
&\leq  \sec^3\alpha\;\| \Re\left(f(A) + f(B)\right)\|\hspace{1.2cm}\text{(by Lemma\; \ref{lemma_f_real_sec_f})}\\
&=\sec^3\alpha\; w\left( \Re\left(f(A) + f(B)\right)\right) \\
&\leq \sec^3\alpha\; w\left( f(A) + f(B)\right) \hspace{1.2cm}\text{(by Lemma\; \ref{nume_real<nemu A})},
\end{align*}
which completes the proof.

\end{proof}

\begin{corollary} \label{cor_f(A+B)}Let $ A, B\in\mathcal{S}_{\alpha}$. Then, for $0<t<1,$
\begin{align}
w((A+B)^t)\leq \sec^{3}\alpha\;w(A^t+B^t).
\end{align}
\end{corollary}
\begin{proof} This is an immediate consequence of Theorem \ref{nemu_f}, by putting $f(x)=x^t,$ for $0\leq t\leq 1$.

When $A,B$ are positive, then $\alpha=0,$ and we obtain the well known inequality $\|(A+B)^t\|\leq \|A^t+B^t\|,$ for $t\in[0,1]$.
\end{proof}

\begin{corollary} Let $ A, B\in\mathcal{S}_{\alpha}$. Then, for $0<t<1,$
\begin{align}
\cos^{3}\alpha\;w((A+B)^t)\leq w(A^t+B^t)\leq 2^{1-t}\sec^3\alpha\;\left( w(A)+w(B)\right)^t. 
\end{align}
\end{corollary}
\begin{proof}
It's result from Corollary \ref{exm} and Corollary \ref{cor_f(A+B)}.
\end{proof}

When $A,B\in\mathcal{M}_n$ are positive, then clearly $w(A+B)\geq \max(w(A),w(B)).$ If either $A$ or $B$ is not positive, this inequality is not necessarily true. However, when $A,B$ are sectorial, we have the following version.

\begin{theorem} Let $ A, B\in\mathcal{S}_{\alpha}$. Then
\begin{align}
\cos^2\alpha\;\max(w(A),w(B))\leq w(A+B).
\end{align}

\end{theorem}
\begin{proof} Let $A, B\in\mathcal{S}_{\alpha}$. Then
\begin{align*}
w(A+B)&\geq \cos\alpha\;\|A+B\| \hspace{2cm} \text{(by Proposition\; \ref{prop_1})}\\
&\geq\cos\alpha\;\|\Re A+\Re B\| \hspace{2cm} \text{(by Lemma\; \ref{norm})}\\
&= 2\cos\alpha\;w\left( \begin{pmatrix}
0&\Re A\\
\Re B&0
\end{pmatrix}\right) \hspace{2cm} \text{(Lemma\;\ref{abu_omar_kittaneh})}\\
&\geq\cos\alpha\;\left\|\begin{pmatrix}
0&\Re A\\
\Re B&0
\end{pmatrix}\right\|\hspace{2cm}\text{(by \; \eqref{eq_num_oper})}\\
&=\cos\alpha\;\max\left( \|\Re A\|,\|\Re B\|\right) \hspace{2cm} \text{(by Lemma\; \ref{max_norm})}\\
&=\cos\alpha\;\max\left( w(\Re A),w(\Re B)\right) \\
&\geq\cos^2\alpha\;\max\left( w(A),w(B)\right),\hspace{2cm} \text{(by \; \ref{inv_w(rA)})}\\
\end{align*}
completing the proof.
\end{proof}

\subsection{The numerical radius and operator mean}
In this part of the paper, we present another new type of numerical radius inequalities, where the numerical radius of operator means is discussed. When $A,B\in\mathcal{M}_n$ are positive, then for any operator mean $\sigma$, one has $$A\sigma B>0\Rightarrow w(A\sigma B)=\|A\sigma B\|.$$ This makes the study of numerical radius inequalities of means of positive matrices trivial.

Now we have the following numerical radius action over the operator mean of sectorial matrices.
\begin{theorem} Let $ A, B\in\mathcal{S}_{\alpha}$. If $ f\in\mathfrak{m}$, then
\begin{align}\label{num_of_sigma}
w(A \sigma_f B)\leq\sec^3\alpha\;\left( w(A)\sigma_f w(B)\right).
\end{align}
\end{theorem}
\begin{proof}
Noting Proposition \ref{prop_asigmab_s}, we have
\begin{align*}
w(A \sigma_f B)&\leq \|A \sigma_f B\|\\
&\leq \sec\alpha\;\|\Re(A \sigma_f B)\|\hspace{1.2cm}\text{(by Lemma\; \ref{norm})}\\
&\leq\sec^3\alpha\;\|\Re(A)\sigma_f \Re(B)\|\hspace{1.2cm}\text{(by Lemma\; \ref{lemma_real_a_sigma_b_less})}\\
&\leq \sec^3\alpha\;(\|\Re(A)\|\;\sigma_f\;\| \Re(B)\|)\hspace{1.2cm}\text{(by Lemma\; \ref{sigma_norm})}\\
&= \sec^3\alpha\;(w(\Re A)\;\sigma_f\;w(\Re B))\\
&\leq \sec^3\alpha\;(w(A)\;\sigma_f\; w(B)),\hspace{1.2cm}\text{(by Lemma\; \ref{nume_real<nemu A}\; and Lemma\; \ref{monoto_mean})}
\end{align*}
which completes the proof.
\end{proof}
In particular, when $A,B\in\mathcal{M}_n$ are positive, then we may select $\alpha=0$, and \eqref{num_of_sigma} implies the known inequality
$$\|A\sigma_fB\|\leq \|A\|\sigma_f\|B\|.$$
\begin{corollary}  Let $ A, B\in\mathcal{S}_{\alpha}$. Then, for $0<t<1,$
\begin{equation}\label{nume_sharp_inq} 
 w(A\sharp_{t}B)\leq \sec^{3}\alpha\;w^{1-t}(A)w^t(B).
\end{equation}
\end{corollary}
\begin{proof} By \eqref{num_of_sigma} and for $\sigma_f =\sharp_t ,$  where $ t\in(1,0)$ 
\begin{align*}
w(A\sharp_{t}B)&\leq \sec^3\alpha\;(w(A)\;\sharp_t\; w(B))\\
&=\sec^{3}\alpha\;w^{1-t}(A)w^t(B),
\end{align*}
where the last inequality, follows from the definition of $\sharp_t$.
This completes the proof.
\end{proof}
In a similar way, the logarithmic mean satisfies similar property, as follows.
\begin{theorem}
Let $ A, B \in \mathcal{S}_{\alpha} $. Then, for $0<t<1,$
\begin{align}
w(\mathcal{L}(A,B))\leq \sec^3\alpha\;\mathcal{L}(w(A),w(B)).
\end{align}
\end{theorem}
\begin{proof} By definition of the logarithmic mean \eqref{loga_matrices}, we get
\begin{align*}
w(\mathcal{L}(A,B))&=w\left( \int^{1}_{0}A\sharp_{t}B\;dt\right) \\
&\leq \int^{1}_{0}w(A\sharp_{t}B)\;dt\\
&\leq \sec^3\alpha\;\int^{1}_{0}w^{1-t}(A)w^t(B)\;dt\hspace{1.2cm}\text{(by\; \eqref{nume_sharp_inq})}\\
&=\sec^3\alpha\;\mathcal{L}(w(A),w(B)),
\end{align*}
completing the proof.
\end{proof}
The Heinz means follow the same theme too.
\begin{theorem} Let $A, B\in \mathcal{S}_\alpha$. Then for $t\in(0,1),$
\begin{align}
w(\mathcal{H}_t(A,B))\leq \sec^3\alpha\;\mathcal{H}_t(w(A),w(B)).
\end{align} 
\end{theorem}
\begin{proof}Compute
\begin{align*}
w(\mathcal{H}_t(A,B))&=w\left(\dfrac{A\sharp_t B+A\sharp_{1-t} B}{2}\right) \hspace{4.2cm}\text{(by \;\eqref{heinz_for_matrices})}\\
&\leq \dfrac{w(A\sharp_t B)+w(A\sharp_{1-t} B)}{2}\\
&\leq  \dfrac{\sec^3\alpha}{2}\left(w^{1-t}(A)w^t(B)+w^{t}(A)w^{1-t}(B)\right) \hspace{1.2cm}\text{(by\; \eqref{nume_sharp_inq})}\\
&=\sec^3\alpha\;\mathcal{H}_t(w(A),w(B)).
\end{align*}
The proof is complete.

\end{proof}
A Heinz-type inequality for the numerical radii of accretive matrices maybe stated as follows.
\begin{theorem} Let $A, B\in \mathcal{S}_\alpha$. Then for $t\in(0,1),$
\begin{align}
\cos^4\alpha\;w(A\sharp B)\leq w(\mathcal{H}_t(A,B))\leq \dfrac{\sec^4\alpha}{2}\;w(A+B).
\end{align}
\end{theorem}
\begin{proof} We prove the first inequality.
\begin{align*}
w(A\sharp B)&\leq\;\|A\sharp B\| \hspace{2.2cm}\text{(by Proposition\; \eqref{prop_1})}\\
&\leq \sec^3\alpha\;\|\mathcal{H}_t(A,B)\|\hspace{2.2cm}\text{(by Lemma\; \eqref{heinz_norm_bound})}\\
&\leq \sec^4\alpha\;w(\mathcal{H}_t(A,B)).\hspace{2.2cm}\text{(by Proposition\; \eqref{prop_1})}
\end{align*}
We now prove the second inequality.
\begin{align*}
w(\mathcal{H}_t(A,B))&\leq \|\mathcal{H}_t(A,B)\|\hspace{2.2cm}\text{(by Proposition\; \eqref{prop_1})}\\
&\leq \sec^3\alpha\;\|\dfrac{A+B}{2}\|\hspace{2.2cm}\text{(by Lemma\; \eqref{heinz_norm_bound})}\\
&\leq \dfrac{\sec^4\alpha}{2}\;w(A+B).\hspace{2.2cm}\text{(by Proposition\; \eqref{prop_1})}
\end{align*}

\end{proof}

The proof is complete.
\begin{corollary} Let $A, B\in \mathcal{S}_\alpha$. Then for $t\in(0,1),$
\begin{align}
\cos\alpha\;w^{\frac{1}{2}}(AB)\leq \mathcal{H}_t(w(A),w(B)).
\end{align}
\end{corollary}

\begin{proof} By Theorem \ref{nemu_for_AB}, we get
\begin{align*}
\cos\alpha\;w^{\frac{1}{2}}(AB)\leq \sqrt{w(A)w(B)}\leq\mathcal{H}_t(w(A),w(B)).
\end{align*}

\end{proof}

{\tiny \vskip 1 true cm }
{\tiny (Y. Bedrani) Department of Mathematics, The University of Jordan, Amman, Jordan. 

\textit{E-mail address:} \bf{yacinebedrani9@gmail.com}}
{\tiny \vskip 0.3 true cm }
 {\tiny (F. Kittaneh) Department of Mathematics, The University of Jordan, Amman, Jordan. 
 
  \textit{E-mail address:} \textit{E-mail address:} \bf{fkitt@ju.edu.jo}}
 {\tiny \vskip 0.3 true cm }
{\tiny (M. Sababheh) Dept. of Basic Sciences, Princess Sumaya University for Tech., Amman 11941, Jordan.
 
\textit{E-mail address:} \bf{sababheh@psut.edu.jo}}

 {\tiny \vskip 0.3 true cm }

\end{document}